\documentclass[11pt]{article}

\usepackage{amsmath, amssymb, amsthm} 

\usepackage{dsfont}

\allowdisplaybreaks
\expandafter\let\expandafter\oldproof\csname\string\proof\endcsname
\let\oldendproof\endproof
\renewenvironment{proof}[1][\proofname]{%
	\oldproof[\bf #1]%
}{\oldendproof}

\allowdisplaybreaks

\theoremstyle{plain}
\newtheorem{theorem}{Theorem}
\newtheorem{lemma}{Lemma}[section]

\newtheorem{construction}[lemma]{Construction}

\RequirePackage[normalem]{ulem} 
\RequirePackage{color}\definecolor{RED}{rgb}{1,0,0}\definecolor{BLUE}{rgb}{0,0,1} 

\newcommand{\ex}{\text{ex}}

\title{Constructing Dense Grid-Free Linear $3$-Graphs}

\author{Lior Gishboliner\thanks{ETH Zurich. 
			Email: lior.gishboliner@math.ethz.ch.}
	\and Asaf Shapira\thanks{
		School of Mathematics, Tel Aviv University, Tel Aviv 69978, Israel.
		Email: asafico@tau.ac.il. Supported in part by ISF Grant 1028/16, ERC Consolidator Grant 863438 and NSF-BSF Grant 20196.}
}

\begin{document}
	\maketitle 
	
	\begin{abstract}
		We show that there exist linear $3$-uniform hypergraphs with $n$ vertices and $\Omega(n^2)$ edges which contain no copy of the $3 \times 3$ grid. This makes significant progress on a conjecture of F\"{u}redi and \nolinebreak Ruszink\'{o}. We also discuss connections to proving lower bounds for the $(9,6)$ Brown-Erd\H{o}s-S\'{o}s problem and to a problem of Solymosi and Solymosi. 
	\end{abstract}
	
	\section{Introduction}
	In recent years there has been some interest in Tur\'{a}n-type results for linear hypergraphs \cite{FR, Ge_Shangguan, GS}. 
	In this paper, all hypergraphs are $3$-uniform. For a family $\mathcal{H}$ of $3$-uniform hypergraphs, we let $\ex_{\text{lin}}(n,\mathcal{H})$ denote the maximum number of edges in a linear $3$-uniform $\mathcal{H}$-free hypergraph on $n$ vertices. When $\mathcal{H}$ has a single element $H$, we will write $\ex_{\text{lin}}(n,H)$. Arguably, the interest in problems of this type is motivated by the famous Brown-Erd\H{o}s-S\'{o}s conjecture \cite{BHS1,BHS2}, which states that, for every $k \geq 3$, if $\mathcal{H}_{k+3,k}$ is the set of all $3$-uniform hypergraphs with $k$ edges and at most $k+3$ vertices (such hypergraphs are called {\em $(k+3,k)$-configurations}), then\footnote{The Brown-Erd\H{o}s-S\'{o}s conjecture is usually stated about general (i.e., not necessarily linear) hypergraphs, but it is well-known that it suffices to consider linear hypergraphs. Indeed, if a hypergraph $H$ contains no $(k+3,k)$-configuration, then every pair of vertices is contained in at most $k-1$ edges, so $H$ has a linear subhypergraph with at least $e(H)/(k-1) = \Omega(e(H))$ edges.} $\ex_{\text{lin}}(n,\mathcal{H}_{k+3,k}) = o(n^2)$. So far, this conjecture has only been proven in the case $k = 3$. This is a celebrated result of Ruzsa and Szemer\'{e}di \cite{RS}, which became known as the $(6,3)$ theorem. Ruzsa and Szemer\'{e}di \cite{RS} have also given a construction which shows that $\ex_{\text{lin}}(n,\mathcal{H}_{6,3}) \geq n^{2-o(1)}$, implying that the exponent $2$ in the $(6,3)$ theorem cannot be improved. 
	For $k \geq 4$, the Brown-Erd\H{o}s-S\'{o}s conjecture remains widely open despite considerable effort, with the best approximate result recently obtained in \cite{CGLS} (see also \cite{Sarkozy_Selkow, Solymosi}).
	
	It is easy to check that $\mathcal{H}_{6,3}$ contains only one linear hypergraph: the triangle $\mathbb{T}$, which is the hypergraph with vertices $1,2,3,4,5,6$ and edges $\{1,2,3\},\{3,4,5\},\{5,6,1\}$. Thus, the aforementioned results of Ruzsa and Szemer\'{e}di \cite{RS} are equivalent to the statement 
	$n^{2-o(1)} \leq \ex_{\text{lin}}(n,\mathbb{T}) \leq o(n^2)$. 
	
	It is natural to try and prove that $\ex_{\text{lin}}(n,\mathcal{H}_{k+3,k}) \geq n^{2-o(1)}$ for every $k \geq 3$, which would mean that, in a sense, the Brown-Erd\H{o}s-S\'{o}s conjecture is optimal. For $k = 4,5$, such a lower bound follows from the simple observation that every $(7,4)$- or $(8,5)$-configuration contains a $(6,3)$-configuration. Similar considerations were used in \cite{Ge_Shangguan} to handle the cases $k = 7,8$. For $k = 6$, however, such arguments could not be used, since there exists a $(9,6)$-configuration which contains no $(6,3)$-configuration; this is the {\em $3 \times 3$ grid} $\mathbb{G}_{3 \times 3}$, which is the $3$-uniform hypergraph whose vertices are the nine points in a $3 \times 3$ point array, and whose edges correspond to the $6$ horizontal and vertical lines of this array.
	It is not hard to verify\footnote{Indeed, let $H$ be a linear $(9,6)$-configuration avoinding $\mathbb{T}$. First, observe that $H$ has maximum degree $2$, for if $\{a,b,c\},\{a,d,e\},\{a,f,g\}$ are three edges containing $a$, then there can be only one edge containing the remaining two vertices (as $H$ is linear), so there must be an edge which contains two vertices from $\{b,c,d,e,f,g\}$, which gives a $\mathbb{T}$. Now, as $e(H) = 6$, all degrees in $H$ must be $2$. Consider the two edges $\{a,b,c\},\{a,d,e\}$ containing some vertex $a$. Let $f,g,h,i$ be the four remaining vertices. Each of the four remaining edges must contain two vertices from $\{f,g,h,i\}$ and one from $\{b,c,d,e\}$. Every vertex from $\{b,c,d,e\}$ must be covered once by these edges, and every vertex from $\{f,g,h,i\}$ twice. Hence, the pairs from $\{f,g,h,i\}$ which are covered by these edges must form a $C_4$. Since $H$ is $\mathbb{T}$-free, $b$ and $c$ must be contained in opposite edges of this $C_4$, and the same for $d$ and $e$. This gives a $\mathbb{G}_{3,3}$.} (see also \cite{Ge_Shangguan}) that every linear $(9,6)$-configuration either contains a triangle $\mathbb{T}$ or is isomorphic to $\mathbb{G}_{3 \times 3}$. Hence, $\ex_{\text{lin}}(n,\mathcal{H}_{9,6}) \geq  \ex_{\text{lin}}(n,\{\mathbb{T},\mathbb{G}_{3 \times 3}\})$. This relation has led F\"{u}redi and Ruszink\'{o} \cite{FR} to study extremal problems related to the grid. In particular, they conjectured that 
	$\ex_{\text{lin}}(n,\mathbb{G}_{3 \times 3}) = 
	\left( \frac{1}{6} - o(1) \right)n^2$, and, more strongly, that for every large enough admissible $n$, there exists a Steiner triple system of order $n$ which is $\mathbb{G}_{3 \times 3}$-free. 
	Using a standard probabilistic alterations argument, F\"{u}redi and Ruszink\'{o} \cite{FR} showed that $\ex_{\text{lin}}(n,\mathbb{G}_{3 \times 3}) = \Omega(n^{1.8})$. This was then slightly improved (as a special case of a more general result) to $\Omega(n^{1.8}\log^{1/5}{n})$ by Shangguan and Tamo \cite{ST}.
	Here we make significant progress on the conjecture of F\"{u}redi and Ruszink\'{o} \cite{FR}, by showing that 
	$\ex_{\text{lin}}(n,\mathbb{G}_{3 \times 3}) = \Omega(n^2)$. 
	\begin{theorem}\label{thm:main}
		For infinitely many $n$, there exists a linear $\mathbb{G}_{3 \times 3}$-free $3$-uniform hypergraph with $n$ vertices and $(\frac{1}{16} - o(1))n^2$ edges. 
	\end{theorem}
	\noindent
	Theorem \ref{thm:main} is proved in the following section. Then, in Section \ref{sec:concluding_remarks}, we discuss some related open problems. 
	
	\section{The Construction}

	\begin{construction}\label{construction}
		Let $\mathbb{F}$ be a field and let $X,A \subseteq \mathbb{F}$. Define $H(X,A)$ to be the $3$-partite $3$-uniform hypergraph with sides $X$, $Y := \{x + a : x \in X, a \in A\}$ and $Z := \{ x \cdot a : x \in X, a \in A \}$, and with an edge $(x,x+a,x \cdot a) \in X \times Y \times Z$ for every $x \in X$ and $a \in A$.    
	\end{construction}
	We now prove that the hypergraph $H(X,A)$ defined in Construction \ref{construction} is always 
	$\mathbb{G}_{3 \times 3}$-free. We will then show that it contains a dense linear subhypergraph.
	We denote the vertices of $\mathbb{G}_{3 \times 3}$ by $\{p_i,q_i,r_i : 1 \leq i \leq 3\}$ and its edges by 
	$\{ \{p_i,q_i,r_i\},\{p_{i+1},q_{i+2},r_{i}\} : 1 \leq i \leq 3\}$, where (here and later on) indices are taken modulo $3$.
	A {\em $3$-partition} of a $3$-uniform hypergraph $F$ is a partition $V(F) = P \cup Q \cup R$ such that every edge of $F$ contains one element from each of the sets $P,Q,R$. 
	Observe that $\{p_1,p_2,p_3\},\{q_1,q_2,q_3\},\{r_1,r_2,r_3\}$ is a $3$-partition of $\mathbb{G}_{3 \times 3}$. It can be verified\footnote{Indeed, every $3$-partition of $\mathbb{G}_{3 \times 3}$ is either obtained from the $3$-partition $(P,Q,R)$ by permuting its classes, or equals $\left( \{p_1,q_3,r_2\},\{p_2,q_1,r_3\},\{p_3,q_2,r_1\} \right)$ or one of its permutations.} that every two $3$-partitions of $\mathbb{G}_{3 \times 3}$ are equivalent, in the sense that there is an automorphism of $\mathbb{G}_{3 \times 3}$ which maps every class of one to a class of the other. 
	\begin{lemma}\label{lem:grid_copy}
		Let $\mathbb{F}$ be a field and let $X,A \subseteq \mathbb{F}$.
		Then $H(X,A)$ is $\mathbb{G}_{3\times 3}$-free. 
	\end{lemma}
	
	\begin{proof}
		Suppose, for the sake of contradiction, that $H(X,A)$ contains a copy of $\mathbb{G}_{3 \times 3}$. 
		Since all $3$-partitions of $\mathbb{G}_{3 \times 3}$ are equivalent (as explained above), we may assume, without loss of generality, that $p_1,p_2,p_3 \in X$, $q_1,q_2,q_3 \in Y = \{x + a : x \in X, a \in A\}$ and $r_1,r_2,r_3 \in Z = \{ x \cdot a : x \in X, a \in A \}$. 
		By definition of $H(X,A)$, for every edge $\{x,y,z\} \in E(H)$ (with $x \in X$, $y \in Y$ and $z \in Z$) there is $a \in A$ such that $y = x + a$ and $z = x \cdot a$; hence, $z = x \cdot (y - x)$. 
		It follows that for every $1 \leq i \leq 3$, \nolinebreak we \nolinebreak must \nolinebreak have 
		$r_i = p_i \cdot (q_i - p_i) \text{ and } r_i = p_{i+1} \cdot (q_{i+2} - p_{i+1}).
		$
		Here and throughout the proof, indices are taken modulo $3$. By comparing these two expressions for $r_i$, we see that 
		\begin{equation}\label{eq:grid_copy_equation_1}
		p_i \cdot (q_i - p_i) = p_{i+1} \cdot (q_{i+2} - p_{i+1}).
		\end{equation}
		for every $1 \leq i \leq 3$. Multiplying \eqref{eq:grid_copy_equation_1} by $p_{i+2}$ and then summing over $1 \leq i \leq 3$, we obtain
		$$
		\sum_{i = 1}^{3}{p_ip_{i+2} \cdot (q_i - p_i)} = 
		\sum_{i = 1}^{3}{p_{i+1}p_{i+2} \cdot (q_{i+2} - p_{i+1})}.
		$$
		It is easy to see that for every $1 \leq i \leq 3$, both sides have the term $p_ip_{i+2}q_i$. Cancelling out these terms and rearranging, we get
		$$
		0 = \sum_{i = 1}^{3}{p_i^2p_{i+2}} - \sum_{i = 1}^{3}{p_{i+1}^2p_{i+2}} = 
		(p_1 - p_2)(p_2 - p_3)(p_3 - p_1).
		$$
		Hence, there must be $1 \leq i \leq 3$ such that $p_{i+1} = p_i$. However, this is impossible as $p_1,p_2,p_3 \in X$ must correspond to distinct vertices of a copy of $\mathbb{G}_{3 \times 3}$. This contradiction completes the proof. 
	\end{proof} 
	\begin{proof}[Proof of Theorem \ref{thm:main}]
		We first prove Theorem \ref{thm:main} with a slightly worse bound, namely, with the fraction $\frac{1}{16}$ replaced by $\frac{1}{18}$. We then explain how our argument can be modified to give $\frac{1}{16}$.
		
		Let $p$ be an odd prime power, and set $X := A := \mathbb{F}_p \setminus \{0\}$. Let $H = H(X,A)$ be the hypergraph from Construction \ref{construction}. By Lemma \ref{lem:grid_copy}, $H$ is $\mathbb{G}_{3 \times 3}$-free. 
		We claim that for each edge $e = (x,x+\nolinebreak a,x\cdot \nolinebreak a) \in E(H) \subseteq X \times Y \times Z$, if $f \in E(H) \setminus \{e\}$ satisfies that $|e \cap f| = 2$ then $f = (a,x+a,x \cdot a)$. So let $f = (y,y+b,y \cdot b) \in E(H) \setminus \{e\}$ be such that $|e \cap f | = 2$. We cannot have $(x,x+a) = (y,y+b)$ or $(x,x \cdot a) = (y,y \cdot b)$, for otherwise we would have $x = y, a = b$ and hence $e = f$. Therefore, we must have $(x+a,x \cdot a) = (y+b, y\cdot b)$, which gives 
		$
		y(x + a - y) = x \cdot a.
		$
		Solving this quadratic equation for $y$, we get that $y = x$ or $y = a$, and hence $(y,b) = (x,a)$ or $(y,b) = (a,x)$. In the former case, $f = e$, and in the latter case $f = (a,x+a,x \cdot a)$. This proves our claim. 
		It follows that for each $e \in E(H)$ there is at most one other edge $f \in E(H)$ such that $|e \cap f| = 2$. By deleting one edge from each such pair $(e,f)$, we obtain a linear sub-hypergraph $H'$ of $H$ with 
		$e(H') \geq \frac{e(H)}{2} = |X||A| = (\frac{1}{2} - o(1))p^2 = (\frac{1}{18} - o(1))v(H)^2$, where in the last equality we used the fact that $v(H) = 3p-2$ as $|X| = p-1$, $|Y| = p$ and $Z = p-1$. This shows that $\ex_{\text{lin}}(n,\mathbb{G}_{3 \times 3}) \geq 
		(\frac{1}{18} - o(1))n^2$.
		
		To improve the constant, we choose $X$ and $A$ differently: let $X$ be the set of (non-zero) quadratic residues and $A$ be the set of (non-zero) quadratic non-residues in $\mathbb{F}_p$. Evidently, $|X| = |A| = \frac{p-1}{2}$ and $|Y| \leq p$. As
		$Z = \{ x \cdot a : x \in X, a \in A \} = A$, one also has $|Z| = \frac{p-1}{2}$. Altogether we get 
		$v(H) = |X| + |Y| + |Z| \leq 2p-1$. Moreover, $e(H) = |X||A| = (\frac{1}{4} - o(1))p^2 = (\frac{1}{16} - o(1))v(H)^2$. Crucially, we observe that $H$ is linear, because for every $e = (x,x+a,x \cdot a) \in E(H)$, the edge $f = (a, x + a, x \cdot a)$ is not in $H$, as $x$ is a quadratic residue while $a$ is not. This completes the proof.
	\end{proof}
\section{Concluding Remarks And Open Problems}\label{sec:concluding_remarks}
\begin{itemize}
\item Another problem raised in \cite{FR} is to prove that 
$\ex_{\text{lin}}(n,\mathcal{H}_{9,6}) \geq n^{2-o(1)}$. This problem remains open. Recalling that $\ex_{\text{lin}}(n,\mathcal{H}_{9,6}) \geq \ex_{\text{lin}}(n,\{\mathbb{T},\mathbb{G}_{3 \times 3}\})$, we see, in light of Lemma \ref{lem:grid_copy}, that it suffices to find a choice of sets $X,A \subseteq \mathbb{F}_p$, $|X|,|A| \geq p^{1-o(1)}$, such that the hypergraph $H(X,A)$ has no triangles (i.e., no copies of $\mathbb{T}$). For this, one needs that there are no $x \in X$ and distinct $a,b,c \in A$ such that $(x + a - b) \cdot b = x \cdot c$. 

\item There is another construction of a linear $3$-uniform grid-free hypergraph with $\Omega(n^2)$ edges. For sets $X,A \subseteq \mathbb{F}_p$, define a $3$-partite hypergraph with sides $X,Y,Z$ by placing the edge \linebreak $(x,x+a,x+a^2) \in X \times Y \times Z$ for every $x \in X, a \in A$. 
Here one needs to be more careful: unlike Construction \ref{construction}, this hypergraph can contain a copy of $\mathbb{G}_{3 \times 3}$, but only if there are $x_1,x_2 \in X$ and $a \in A$ satisfying $4x_1 + 4a = 4x_2 + 1$. 
Let us prove this. Consider a copy of $\mathbb{G}_{3,3}$ with vertices $\{p_i,q_i,r_i : 1 \leq i \leq 3\}$, as described before Lemma \ref{lem:grid_copy}. Here, this copy corresponds to the equations $r_i - p_i = (q_i - p_i)^2$ and $r_i - p_{i+1} = (q_{i+2} - p_{i+1})^2$ for $i = 1,2,3$. Hence, $p_i + (q_i - p_i)^2 = p_{i+1} + (q_{i+2} - p_{i+1})^2$. Substituting $u_i := p_{i+1} - p_i$ and $v_i := q_i - p_{i+1}$ ($i = 1,2,3$), we get $(v_i + u_i)^2 = u_i + (v_{i+2} - u_i)^2$, and, after rearranging, 
\begin{equation}\label{eq:quadratic_main}
(2v_i + 2v_{i+2} - 1)u_i = v_{i+2}^2 - v_i^2.
\end{equation}

Now, if $2v_i + 2v_{i+2} \neq 1$ for all $1 \leq i \leq 3$, then in equation \eqref{eq:quadratic_main} we can divide and get $u_i = (v_{i+2}^2 - v_i^2)/(2v_i + 2v_{i+2} - 1)$ for all $1 \leq i \leq 3$. Summing this over $i$ and using the fact that $u_1 + u_2 + u_3 = (p_2 - p_1) + (p_3 - p_2) + (p_1 - p_3) = 0$, we get
$$
0 = 
\sum_{i=1}^3 {u_i} 
= 
\sum_{i=1}^3 {\frac{v_{i+2}^2 - v_i^2}{2v_i + 2v_{i+2} - 1}}
=
\frac{-2(v_3-v_1)(v_1-v_2)(v_2-v_3)}{(2v_1+2v_3-1)(2v_2+2v_1-1)(2v_3+2v_2-1)} \; .
$$
Hence, there must be $1 \leq i \leq 3$ such that $v_{i+2} = v_i$. Plugging this into \eqref{eq:quadratic_main} and using that $2v_i + 2v_{i+2} \neq 1$, we get that $u_i = p_{i+1} - p_i = 0$, which is impossible as $p_i,p_{i+1}$ are distinct vertices. Therefore, there must be $1 \leq i \leq 3$ such that $2v_i + 2v_{i+2} = 1$, hence also $v_{i+2}^2 - v_i^2 = 0$ by \eqref{eq:quadratic_main}. Plugging $v_{i+2} = 1/2 - v_i$ into $v_{i+2}^2 - v_i^2 = 0$, we get that $v_i = 1/4$, hence $q_i - p_{i+1} = 1/4$. Now, recall that by construction, $p_i,p_{i+1} \in X$ and $q_i = p_i + a$ for some $a \in A$. Hence, we have our desired solution to $4x_1 + 4a = 4x_2 + 1$ with $x_1,x_2 \in X$, $a \in A$. 
So in order for the hypergraph to be $\mathbb{G}_{3 \times 3}$-free, it suffices to choose $X,A$ that avoid such solutions; for example, one can take $X = A = \{1,\dots,\lfloor p/8 \rfloor \}$. 

This construction can also be a candidate for showing that $\ex_{\text{lin}}(n,\mathcal{H}_{9,6}) \geq n^{2-o(1)}$.
Again, the issue is choosing $X,A$ so as to avoid triangles, which in this case correspond to solutions to the equation $a + c^2 - c = b^2$ with distinct $a,b,c \in A$. Thus, in order to show that $\ex_{\text{lin}}(n,\mathcal{H}_{9,6}) \geq n^{2-o(1)}$, it suffices to show that there exists $A \subseteq \mathbb{F}_p$, $|A| = p^{1-o(1)}$, with no non-trivial solution to this equation.   
 

\item A related conjecture of Solymosi and Solymosi \cite{Solymosi} states that 
every (large enough) $3$-uniform hypergraph with $n$ vertices and $\Omega(n^2)$ edges contains a {\em $2$-core} on at most $9$ vertices, 
where a $2$-core is a hypergraph with minimum degree $2$. This conjecture is closely related\footnote{Strictly speaking, the Solymosi-Solymosi conjecture does not imply the case $k=6$ of the Brown-Erd\H{o}s-S\'{o}s conjecture, since the former allows the $2$-core to have less than $9$ vertices, and hence less than $6$ edges.} to the case $k=6$ of the Brown-Erd\H{o}s-S\'{o}s conjecture, since a $2$-core on $9$ vertices has at least $6$ edges.


Let $H$ be the $3$-partite hypergraph with sides $X,Y,Z$, all equal to  $\mathbb{F}_p$, and with edge-set 
\linebreak $\{(x,x+a,x+2a) \in X \times Y \times Z : x,a \in \mathbb{F}_p\}$. 
Alternatively, this is the hypergraph whose edges are all triples $(x,y,z) \in X \times Y \times Z$ satisfying $y = (x+z)/2$. 
By a somewhat lengthy case analysis, one can show that $H$ avoids all $2$-cores on at most $9$ vertices except for the grid $\mathbb{G}_{3 \times 3}$. 
Thus, the hypergraph corresponding to a linear relation (namely, the relation $y = (x+z)/2$) avoids all but one of the $2$-cores on at most $9$ vertices, whereas in order to avoid $\mathbb{G}_{3 \times 3}$ one needs a non-linear relation (as in Construction \ref{construction} or in the construction described in the previous item). It would be interesting to understand the connection between the structure of a configuration $F$ and the relation which can be used to define a hypergraph which avoids $F$.

We note that inspite of the above construction, it is plausible that the Solymosi-Solymosi conjecture is true; namely, that while there exist dense linear hypergraphs which avoid any individual $2$-core on at most $9$ vertices (and even hypergraphs which avoid all but one of them), avoiding all such $2$-cores in a dense linear hypergraph is impossible.  
\end{itemize}

\end{document}